\documentclass[11pt]{amsart}

\usepackage{mathpazo}
\usepackage[top=25mm, bottom=25mm, left=25mm, right=25mm]{geometry} 

\usepackage{amsmath,amssymb, amscd, color}
\usepackage{graphicx}
\usepackage{import}
\usepackage{amscd}
\usepackage{wrapfig}
\usepackage{epsfig}
\numberwithin{equation}{section}
\usepackage[alphabetic,nobysame]{amsrefs}
\usepackage{float}
\usepackage{tikz}
\usetikzlibrary{positioning}
\usepackage{comment}

\newtheorem{theorem}{Theorem}[section]

\newtheorem{lemma}[theorem]{Lemma}
\newtheorem{corollary}[theorem]{Corollary}
\newtheorem{proposition}[theorem]{Proposition}

\newtheorem{remark}[theorem]{Remark}

\newcommand{\teichmuller}{Teichm{\"u}ller{ }}

\makeatletter
 \let\c@theorem=\c@subsection
 \let\c@conjecture=\c@subsection
 \let\c@lemma=\c@subsection
 \let\c@proposition=\c@subsection
 \let\c@claim=\c@subsection
 \let\c@question=\c@subsection
 \let\c@criterion=\c@subsection
 \let\c@vfconj=\c@subsection
 \let\c@definition=\c@subsection
 \let\c@notation=\c@subsection
 \let\c@remark=\c@subsection
 \let\c@example=\c@subsection
 \let\c@equation=\c@subsection
 \let\c@figure=\c@subsection
 \let\c@wrapfigure=\c@subsection

\makeatother

\begin{document}
\title{Cusp excursions of random geodesics in Weil--Petersson type metrics}

\author[Gadre]{Vaibhav Gadre}
\address{\hskip-\parindent
     School of Mathematics \& Statistics,
     University of Glasgow,
     University Place,
      Glasgow G12 8QQ UK}
\email{Vaibhav.Gadre@glasgow.ac.uk}

\author[Matheus]{Carlos Matheus}
\address{\hskip-\parindent
  Centre de Math\'{e}matiques Laurent Schwartz, CNRS (UMR 7640), \'{E}cole Polytechnique, 91128 Palaiseau, France.}
\email{carlos.matheus@math.cnrs.fr }
 

\begin{abstract}
We analyse cusp excursions of random geodesics for Weil--Petersson type incomplete metrics on orientable surfaces of finite type: in particular, we give bounds for maximal excursions. 

We also give similar bounds for cusp excursions of random Weil--Petersson geodesics on non-exceptional moduli spaces of Riemann surfaces conditional on the assumption that the Weil--Petersson flow is polynomially mixing.

Moreover, we explain how our methods can be adapted to understand almost greasing collisions of typical trajectories in certain slowly mixing billiards. 
\end{abstract}

\maketitle

\section{Introduction}

Let $S$ be an orientable surface of finite type with cusps. Suppose $S$ is endowed with a \texttt{"}Weil--Petersson\texttt{"} type metric: a negatively curved Riemannian metric which in  neighbourhood of each cusp is asymptotically modelled on a surface of revolution obtained by rotating the curve $y = x^r$ for some $r> 2$ about the $X$-axis in $\mathbb{R}^3$ (where $r$ may depend on the puncture). An example is the Weil--Petersson metric on the modular surface $\mathbb{H}^2/SL(2, \mathbb{Z})$ whose cusp neighbourhood is asymptotically modelled on the surface of revolution of $y = x^3$. We will state the precise hypothesis on the metric in due course. 

Pollicott--Weiss showed that the geodesic flow for a Weil--Petersson type metrics is ergodic \cite[Theorem 5.1]{Pol-Wei}. In more recent work, Burns--Masur--Matheus--Wilkinson \cite[Theorem 1]{BMMW} show that the geodesic flow for Weil--Petersson type metrics is exponentially mixing. In this article, we analyse cusp excursions of random geodesics proving bounds for \texttt{"}maximal\texttt{"} excursions, a specific example of a shrinking target result. Other well known examples of such results include the logarithm law of Sullivan for a hyperbolic metric on $S$ \cite[Theorem 2]{Sul}, the logarithm law of Kleinbock--Margulis on homogenous spaces \cite{KM} and the logarithm law of Masur for moduli spaces of Riemann surfaces \cite{Mas}.  

We also derive conditional bounds for maximal excursions of random Weil--Petersson geodesics on non-exceptional moduli spaces of Riemann surfaces. Burns--Masur--Wilkinson \cite{BMW} showed that the Weil--Petersson flow is ergodic; Burns--Masur--Matheus--Wilkinson showed that the Weil--Petersson flow is at most polynomially mixing \cite{BMMW2}. We derive our bounds for maximal excursions conditional on the assumption that the rate of mixing is polynomial in time. Moreover, we explain how our methods in the general case can be adapted to understand almost greasing collisions of typical trajectories in certain slowly mixing billiards. 

\subsection{Weil--Petersson type metrics} 
We recall some facts about the metric on a surface of revolution $R$ for $y = x^r$ for $r>2$. The surface $R$ carries a negatively curved, incomplete Riemannian metric. Let $\delta: R \to \mathbb{R}_{\geqslant 0}$ denote the Riemannian distance to the cusp. The Gaussian curvature 
at $p \in R$ has the following asymptotic expansion in $\delta$ as $\delta \to 0$:
\begin{equation}\label{e.curvature}
K(p) = - \frac{r(r-1)}{\delta(p)^2} + O\left( \frac{1}{\delta(p)} \right).
\end{equation}

We consider negatively curved Riemannian metrics on $S$ that have singularities of the above form. Let $X$ be a closed orientable surface and let $\{p_1, \dots, p_k\}$ be a collection of distinct points in $X$. A Weil--Petersson type metric $\rho$ is a $C^5$, negatively curved Riemannian metric carried by the punctured surface $S = X \setminus \{ p_1, \dots, p_k \}$ such that
\begin{enumerate}
\item the metric $\rho$ extends to a complete metric on $X$,
\item the metric $\rho$ lifts to a geodesically convex metric on the universal cover of $S$, and 
\item if $\delta_j: S \to \mathbb{R}_+$ is the distance to $p_j$ then there exists $r_1, \dots, r_k > 0$ such that the Gaussian curvature satisfies
\[
K(p) = -\sum\limits_{j=1}^k \frac{r_j(r_j-1)}{\delta_j(p)^2} + O\left( \frac{1}{\delta_j(p)} \right)
\]
and if $\nabla$ is the Levi-Civita connection defined by the metric then 
\[
\Vert \nabla^m(K(p)) \Vert = \sum\limits_{j=1}^{k} O\left( \frac{1}{\delta_j^{2+m}(p)}\right)
\]
for $m = 1, 2,3$ and all $p \in S$. 

\end{enumerate} 
Let $\mu$ be the Liouville measure on the unit tangent bundle $T^1S$ for the geodesic flow $\phi_t: T^1S \to T^1 S$ associated to a Weil--Petersson type metric. The main theorem we prove is:

\begin{theorem}
For all $\varepsilon>0$ and for $\mu$-almost every $v \in T^1S$, there is $T(v,\varepsilon) > 0$ such that the geodesic $\phi_t v$ satisfies
\[
\frac{1}{c T^{(1+\varepsilon)/r_j}} \leqslant \inf\limits_{t \leqslant T} \, \delta_j(\phi_t v) \leqslant \frac{c}{T^{(1-\varepsilon)/2r_j}}
\]
for all $j = 1, \dots, k$ and for all $T > T(v,\varepsilon)$.
\end{theorem}

\subsection{Exponential mixing}
The main tool we use is the exponential mixing of the flow \cite{BMMW}. For $\theta \in (0,1]$ and an $L^\infty$ function $f: T^1S \to \mathbb{R}$ let 
\[
\vert f \vert_\theta = \sup\limits_{v \neq v'} \frac{\vert f(v) - f(v') \vert}{\rho(v, v')^\theta}.
\]
Let $C^\theta(T^1S)$ be the set of $L^\infty$ functions $f: T^1 S \to \mathbb{R}$ such that 
$\Vert f \Vert_\theta$ defined as  $\Vert f \Vert_\theta = \vert f \vert_{L^\infty} + \vert f \vert_\theta$ is finite. The precise decay of correlations is:

\begin{theorem}[Burns--Masur--Matheus--Wilkinson]
For every value $\theta \in (0,1]$ there are constants $K, C > 0$ such that for every $f_1, f_2 \in C^\theta(T^1 S)$ we have
\[
\left\vert \int\limits_{T^1 S} f_1(v) f_2(\phi_t v) \, d\mu - \int\limits_{T^1 S} f_1(v) \, d\mu \int\limits_{T^1 S} f_2(v) \, d\mu \right\vert \leqslant K e^{-Ct} \Vert f_1 \Vert_\theta \Vert f_2 \Vert_\theta.
\]
\end{theorem}

\subsection{Organisation of the paper} 

Using the exponential mixing of Weil--Petersson type flows, in Section \ref{s.effective} we first obtain effective convergence of ergodic averages. After that, we review in Section \ref{s.BMMW-review} some useful facts about the geometry of cusps of Weil--Petersson type metrics. The effective theory and the cusp geometry are then used in Section \ref{s.excursions} to estimate the time spent in a cusp neighbourhood of a particular depth, so that an excursion to this neighbourhood happens or is prohibited depending on whether the estimate is positive or not. Next, we study in Section \ref{s.winding} how random Weil--Petersson type geodesics do wind near cusps and the hyperbolic distances covered by (genuine) Weil--Petersson geodesics in exceptional moduli spaces. Finally, we briefly explain in Section \ref{s.applications} certain extensions of some of our results to the case of slowly mixing flows including a class of billiards. For non-exceptional moduli spaces, these extended results give bounds for cusp excursions conditional on the assumption that the Weil--Petersson flow has a polynomial mixing rate that is at least linear in time.

\subsection*{Acknowledgements} We are thankful to Jonathan Chaika for some discussions related to this work which occurred during a visit of the first author to him at the Institut Henri Poincar\'e (IHP). Also, the first author would like to thank the IHP for its hospitality. We thank the referees for their comments which helped to improve the paper.

\section{Ergodic averages of exponentially mixing flows}\label{s.effective}

Let $g_t$ be a flow on a space $Y$ preserving a probability measure $\mu$. Suppose that $g_t$ has an exponential decay of correlations i.e., there are constants $K, C>0$ such that 
\begin{equation}\label{e.correlation}
\left\vert \int\limits_Y f_1 f_2 \circ g_t \, d\mu - \int\limits_Y f_1 \, d\mu \int\limits_Y f_2 \, d\mu \right\vert \leqslant K e^{-Ct} \Vert f_1 \Vert_B \Vert f_2 \Vert_B 
\end{equation}
for all $t \geqslant 0$ and all real valued observables $f_1, f_2$ in a Banach space $B \subset L^1(\mu)$ consisting of all functions satisfying some regularity hypothesis.  

\begin{lemma}\label{l.1} 
Any observable $ f \in B$ with $\int_Y f \, d\mu = 0$ satisfies 
\[
\int\limits_Y \left(\int\limits_0^T f(g_t y) \, dt \right)^2\,d\mu(y) \leqslant \frac{2K}{C} T \|f\|_B^2.
\]
\end{lemma}

\begin{proof} We write 
\begin{align*}
\int\limits_Y \left(\int\limits_0^T f(g_t y) \, dt \right)^2\,d\mu(y) &= \int\limits_Y \int\limits_0^T \int\limits_0^T f(g_t y) f(g_s y) \, dt \, ds \, d\mu(y) \\ 
&= \int\limits_0^T\int\limits_0^T \left(\int_{Y} f(g_t x) f(g_s x) \, d\mu(x)\right) \, dt \, ds .
\end{align*}
By $g_t$-invariance of $\mu$, we get 
\[
\int\limits_Y \left(\int\limits_0^T f(g_t x) \, dt \right)^2\,d\mu(x) = \int\limits_0^T\int\limits_0^T \left(\int\limits_Y f(g_{|t-s|} x) f(x) \, d\mu(x)\right) \, dt \, ds.
\]
The exponential decay of correlations \eqref{e.correlation} implies that 
\[
\int\limits_Y \left(\int\limits_0^T f(g_t x) \, dt \right)^2\,d\mu(x) \leqslant K\| f \|_{B}^2 \int\limits_0^T\int\limits_0^T e^{-C |t-s|} \, dt \, ds\leqslant \frac{2K}{C}T \| f \|_{B}^2.
\]
This proves the lemma. 

\end{proof}

\begin{remark} In Section \ref{s.applications}, we extend this lemma to polynomially mixing flows. 
\end{remark}

\subsection{Effective ergodic theorem for fast mixing flows}
Suppose that $g_t$ is an exponentially mixing flow on $(Y,\mu)$ satisfying \eqref{e.correlation}. 

Fix $1/2<\alpha<1$ and denote $T_k=T_k(\alpha)=k^{2\alpha/(2\alpha-1)}$. 

\begin{theorem}\label{t.effective} 
Given $m>1$, a function $n:\mathbb{R}\to\mathbb{N}$ such that $n(T)=n(T_k)$ for each $T_k \leqslant T < T_{k+1}$, and a sequence $\{f_j\}_{j\in\mathbb{N}}\subset B$ of non-negative functions, we have for $\mu$-almost every $y \in Y$ that 
\begin{align*}
\frac{1}{m} T \| f_{n(T)} \|_{L^1} - 2 T^{\alpha}\| f_{n(T)}\|_{B} &\leqslant \int\limits_0^T f_{n(T)}(g_t y) \, dt \\
&\leqslant mT\| f_{n(T)} \|_{L^1} + 2 T^{\alpha}\| f_{n(T)} \|_B
\end{align*} 
for all $T$ sufficiently large (depending on $y$). 
\end{theorem}

\begin{proof} Given $f\in B$, let $F=f - \int_Y f \, d\mu\in B$. Since $\| F\|_{B} \leqslant 2 \| f \|_{B}$, we get from Lemma \ref{l.1} that  
\[
\int\limits_Y \left(\int\limits_0^T F(g_t y) \, dt \right)^2\,d\mu(y) \leqslant \frac{2K}{C} T \| F \|_{B}^2 \leqslant \frac{8K}{C} T \| f \|_{B}^2.
\]
Therefore, 
\[
\mu \left( \left\{ y \in Y: \left(\int\limits_0^T F(g_t y) \, dt \right)^2 \geqslant R\right\} \right) \leqslant \frac{1}{R} \left( \frac{8K}{C} T \| f \|_{B}^2 \right).
\]
By setting $R=T^{2\alpha} \| f \|_{B}^2$, we obtain 
\begin{equation}\label{e.average}
\mu\left(\left\{y \in Y: \left(\int_0^T F(g_t y) \, dt \right)^2 \geqslant T^{2\alpha}\| f \|_{B}^2\right\}\right)\leqslant \frac{8K}{C} T^{1-2\alpha}.
\end{equation}
Consider the sequence $\{f_j\}_{j\in\mathbb{N}}\subset B$ and let $F_j:= f_j-\int_Y f_j\,d\mu$. From the estimate \eqref{e.average} with $T=T_k$ and $F=F_{n(T_k)}$, and $T=T_{k+1}$ and $F=F_{n(T_k)}$, we get 
\begin{align*} 
\mu \left( \left \{ y \in Y: \left(\int\limits_0^{T_k} F_{n(T_k)}(g_t y) \, dt \right)^2 \geqslant T_k^{2\alpha}\| f_{n(T_k)} \|_{B}^2\right\} \right) &\leqslant \frac{8K}{C} T_k^{1-2\alpha} \\
&= \frac{8K}{C} \frac{1}{k^{2\alpha}}
\end{align*}
and 
\[
\mu \left( \left\{ y \in Y: \left(\int\limits_0^{T_{k+1}} F_{n(T_k)}(g_t y) \, dt \right)^2 \geqslant T_{k+1}^{2\alpha}\| f_{n(T_k)} \|_{B}^2 \right\} \right) 
 \leqslant \frac{8K}{C} \frac{1}{(k+1)^{2\alpha}}.
 \]
By the Borel--Cantelli lemma,
the summability of the series $\sum\limits_{i=1}^{\infty}\frac{1}{i^{2\alpha}}<\infty$ for $\alpha>1/2$ and the previous inequalities imply that for $\mu$-almost every $y \in Y$
\[
\left( \int\limits_0^{T_k} F_{n(T_k)}(g_t y) \, dt \right)^2 \leqslant T_k^{2\alpha}\| f_{n(T_k)} \|_{B}^2
\]
and 
\[
\left(\int\limits_0^{T_{k+1}} F_{n(T_k)}(g_t y) \, dt \right)^2 \leqslant T_{k+1}^{2\alpha}\|f_{n(T_k)}\|_{B}^2
\]
for all $k$ sufficiently large (depending on $y$). 
On the other hand, the non-negativity of the functions $f_j$ says that 
\[
\int\limits_0^{T_k} f_{n(T_k)}(g_t y) \, dt \leqslant \int\limits_0^{T} f_{n(T_k)}(g_t y) \, dt \leqslant \int\limits_0^{T_{k+1}} f_{n(T_k)}(g_t y) \, dt 
\]
for all $T_k \leqslant T < T_{k+1}$. 
Hence, 
\begin{align*}
\int\limits_0^{T_k} F_{n(T_k)}(g_t y)\, dt &= \int\limits_0^{T_k} f_{n(T_k)}(g_t y)\, dt - T_k\int\limits_Y f_{n(T_k)} \, d\mu \\ 
&\leqslant \int\limits_0^{T} f_{n(T_k)}(g_t y)\, dt - T_k\int\limits_Y f_{n(T_k)} \, d\mu \\ 
&\leqslant \int\limits_0^{T_{k+1}} f_{n(T_k)}(g_t y)\, dt - T_k\int\limits_Y f_{n(T_k)} \, d\mu \\ 
&= \int\limits_0^{T_{k+1}} f_{n(T_k)}(g_t y)\, dt 
\begin{aligned}[t] 
&- T_{k+1}\int\limits_Y f_{n(T_k)} \, d\mu \\
&+ (T_{k+1}-T_k)\int\limits_Y f_{n(T_k)} \, d\mu 
\end{aligned}
\\ 
&= \int\limits_0^{T_{k+1}} F_{n(T_k)}(g_t y)\, dt + (T_{k+1}-T_k)\int\limits_Y f_{n(T_k)} \, d\mu
\end{align*}
for all $T_k \leqslant T < T_{k+1}$. We may now combine the above lower and upper bounds on $\int_0^T f_{n(T_k)} (g_t y) \, dt$ with earlier bounds to conclude that for $\mu$-almost every $y \in Y$ and all $k$ sufficiently large (depending on $y$) 
\begin{align*}
T_k \| f_{n(T_k)} \|_{L^1} - T_k^{\alpha} \| f_{n(T_k)}\|_B &\leqslant \int_0^T f_{n(T_k)}(g_t y) \, dt \\
&\leqslant T_{k+1} \| f_{n(T_k)} \|_{L^1} + T_{k+1}^{\alpha} \| f_{n(T_k)}\|_B
\end{align*} 
whenever $T_k \leqslant T < T_{k+1}$. Because $\frac{T_{k+1}}{T_k} = \left(\frac{k+1}{k}\right)^{2\alpha/(2\alpha-1)}\to 1$ as $k \to \infty$ and $n(T)=n(T_k)$ for $T_k \leqslant T<T_{k+1}$, given $m>1$, the previous estimate says that for $\mu$-almost every $y \in Y$
\begin{align*}
\frac{1}{m}T \|f_{n(T)} \|_{L^1} - 2T^{\alpha} \|f_{n(T)} \|_B &\leqslant \int\limits_0^T f_{n(T)}(g_t y) \, dt \\
&\leqslant mT \|f_{n(T)}\|_{L^1} + 2T^{\alpha}\|f_{n(T)}\|_B
\end{align*} 
for all $T$ sufficiently large (depending on $y$ and $m>1$). This proves the theorem. 

\end{proof}

\begin{remark} In Section \ref{s.applications}, we generalise this theorem to polynomially mixing flows. 
\end{remark}

\section{Cusp geometry}\label{s.BMMW-review}

For the sake of expositional simplicity, we will focus on a single cusp neighbourhood and drop the index in the notation. Here, we will recall from \cite{BMMW} properties of Weil-Petersson type metrics that we need to use. The main focus will be commonalities with surfaces of revolution.

\subsection{Local co-ordinates near cusps} Let $W$ be the surface of revolution for $y = x^r$ with $r>2$. In $\mathbb{R}^3$, the surface $W$ inherits the co-ordinates
\[
(x, \tau) \to (x, x^r \cos \tau, x^r \sin \tau) \hskip 5pt \text{ where } \hskip 5pt x \in (0,1], \tau \in [0,2\pi) 
\]
If $\delta$ denotes the distance on $W$ to the cusp $(0,0,0)$ then $\delta(x, \tau) = x+ o(x)$ and the curvature estimate \eqref{e.curvature} holds. Here are some other properties that hold:
\begin{enumerate}
\item The area of the region $\delta \leqslant B$ as $B \to 0$ is 
\[
\frac{2 \pi}{r+1} B^{r+1} + O(B^{r+2}). 
\]
If we let $\ell(\delta_0)$ be the length of the curve given by setting $\delta = \delta_0$ then the constant $2 \pi$ here can be interpreted as the ratio $\ell(\delta_0)/\delta_0^r$ when $\delta_0 = 1$. 
\item \emph{Clairaut integral:} Let $\phi_t(v)$ be a geodesic segment in $T^1 W$ and $\psi_t$ be the angle between the vector $\phi_t(v)$ and the tangent to the foliation given by $\tau = \text{const}$. Then $\phi_t(v) \to  x(\phi_t(v))^r \sin \psi_t$ is a constant function. 
\end{enumerate} 

A Weil-Petersson type metric on $S$ satisfies analogues of the above properties in neighbourhoods of the cusps. As stated before, we will focus on a single cusp neighbourhood and drop the subscript from the notation for the cusp, we will regard the neighbourhood as defined by $\delta \leqslant \delta_0$, where $\delta$ is the distance to the cusp. For $0 < B \leqslant \delta_0$, let $\mathbb{D}^\ast(B)$ be given by $\delta \leqslant B$. Then, the volume of the region $\mathbb{D}^\ast(B)$ satisfies \cite[Section 3.1]{BMMW}
\begin{equation}\label{e.volume}
\text{vol}(\mathbb{D}^\ast (B)) = \frac{\ell(\delta_0) B^{r+1}}{(r+1) \delta_0^r} + O(B^{r+2}) 
\end{equation}
which is an exact analogue for the area term for a surface of revolution. 

To get an analogue of the Clairaut integral, we need a cuspidal angular function that is the analogue of $\sin \theta$ above. Let $J: T^1 S \to T^1 S$ be the almost complex structure compatible with the metric. For a point $p$ in the cusp neighbourhood, let $V \in T^1_p S$ be $\nabla \delta$. By \cite[Corollary 2.2]{BMMW}, $\Vert V \Vert = 1$. For a vector $v \in T^1_p S$, let $a(v) = \langle v, V \rangle$ and $b(v) = \langle v, J V \rangle$. The functions $a$ and $b$ are the analogues of the functions $\cos \psi$ and $\sin \psi$ that appears in the Clairaut integral. We first reproduce from \cite{BMMW} properties of the various quantities that we will need in our analysis.

\begin{lemma}\label{l.convex}\cite[Lemma 3.3]{BMMW}
For every $v \in T^1(\mathbb{D}^\ast(\delta_0))$ such that $a(v) \notin \{1, -1\}$, the function $t \to \delta(\phi_t(v))$ is strictly convex.
\end{lemma}
In particular, the convexity implies
\begin{corollary}\label{c.time}\cite[Corollary 3.4]{BMMW}
Every unit speed geodesic that enters the neighbourhood $\mathbb{D}^\ast(B)$ leaves it in time $\leqslant 2B$. 
\end{corollary}

\subsubsection*{Quasi-Clairaut Relations:} We now state the quasi-Clairaut relations satisfied by the metric.

\begin{proposition}\label{p1.qC}\cite[Proposition 3.2]{BMMW}
If $\delta_0$ is small enough then there exists a constant $M = 1 + O(\delta_0)$ such that for every geodesic segment $\phi(v): [0,T] \to T^1 \mathbb{D}^\ast(\delta_0)$ and all $s, t \in [0,T]$
\[
\frac{1}{M} \delta(\phi_t(v))^r b(\phi_t(v)) \leqslant \delta(\phi_s(v))^r b(\phi_s(v)) \leqslant M \delta(\phi_t(v))^r b(\phi_t(v)).
\]
\end{proposition}

\section{Excursions}\label{s.excursions}

\subsection{Initial positions for excursions}
Let $\delta_0 > 0$ and $0 < d < \delta_0$ be small enough constants such that the collar $C$ given by $\delta \in (\delta_0 - d, \delta_0 + d)$ is contained in the cusp neighbourhood. Given a parameter $R$, we set 
\[
X_R = \left\{ v \in T^1 S \text{ such that } \delta(v) \in C \text{ and } b(v) \leqslant \frac{1}{R} \right\}.
\]
It follows that $\text{vol}(X_R) \asymp 1/R$. 

\begin{proposition}\label{p.depth}
We can fix $\delta_0> 0$ to be small enough such that for any $v \in X_R$, 
\[
\delta(\phi_t(v)) < \frac{2}{R^{1/r}}
\] 
for some time $0 \leqslant t \leqslant a$ where $a=a(r)$ depends only on $r$. 
\end{proposition} 

\begin{proof} By the quasi-Clairaut relations, 
\[
\delta(\phi_t(v))^r b(\phi_t(v)) \leqslant M \delta(v)^r b(v).
\]
Note that $\delta(\phi_t(v))$ is minimised when $b(\phi_t(v))$ is maximised and is roughly 1. 
As the cusp is asymptotically a surface of revolution, we may assume that $\delta_0 < 1$ and it is small enough so that
\begin{itemize}
\item the maximum of $b(\phi_t(v))$ is at least 1/2,
\item $(\delta_0 + d)^r < 2 \delta^r_0$, and
\item the constant $M$ in the quasi-Clairaut relation satisfies $M < 2^{r-2}$. 
\end{itemize} 
We then conclude that at the instant $t_0$ when this maximum is achieved
\[
\delta(\phi_{t_0}(v))^r \leqslant 2M \delta(v)^r b(v) \leqslant \frac{4M \delta_0^r}{R} < \frac{2^r} {R}. 
\] 
By Corollary \ref{c.time}, we have $t_0 \leqslant 2\delta_0$ finishing the proof of the proposition.

\end{proof}

\subsection{Smooth approximations of characteristic functions}
Let $p$ be a smooth non-negative bump function that is 1 on $\delta_0 - d/2 \leqslant \delta \leqslant \delta_0 + d/2$ and supported on $\delta_0 - d \leqslant \delta \leqslant \delta_0 + d$ such that $\| p \|_{C^1} \leqslant 10/d$. Similarly, let $q_R$ be a smooth non-negative bump function 
that is equal to 1 on $b(v) \leqslant 1/2R$ and supported on $b(v) \leqslant 1/R$ such that $\| q_R \| _{C^1}\leqslant 3R$. The non-negative function
\[
f_R(v) = p\left( \delta(v) \right) q_R(b(v))
\]
is a smooth approximation of the characteristic function of $X_R$. Note that
\begin{enumerate}
\item $f_R$ is supported on $X_R$ and
\item there exists a constant $h = h(r) > 1$ depending only on $r$ such that 
\[
\frac{1}{h} \leqslant R \int\limits_S f_R \, d\mu \leqslant h 
\]
and 
\[
\| f_R \|_\theta \leqslant h R^\theta.
\]
\end{enumerate}

\subsection{Deep excursions of typical geodesics} 
At this point, we are ready to use the effective ergodic theorem to show that typical geodesics perform deep cusp excursions: 

\begin{theorem}\label{t.lower} For $\mu$-almost every $v\in T^1 S$, any $\kappa >0$ and for all $T$ sufficiently large (depending on $v$ and $r>2$), 
\[
\delta(\phi_t(v)) \leqslant T^{-\frac{1}{2r}+ \kappa}
\]
for some time $0 \leqslant  t \leqslant T$. 
\end{theorem} 

\begin{proof} Fix $\frac{1}{2}<\alpha<1$, $m=2$, $\theta>0$. Let $\xi>0$ be a parameter to be chosen later and consider the function $n:\mathbb{R}\to\mathbb{N}$, $n(T)=T_k^{\xi}$ for $T_k\leqslant T < T_{k+1}$ (where $T_j:=j^{2\alpha/(2\alpha-1)}$).

The effective ergodic theorem (cf. Theorem \ref{t.effective}) applied to the functions $f_R$ introduced in the previous subsection says that, for $\mu$-almost every $v\in X$ and all $T$ sufficiently large (depending on $v$ and $r>2$),
\[
\int\limits_0^T f_{n(T)}(\phi_t v) \, dt \geqslant \frac{1}{2}T\|f_{n(T)}\|_{L^1} - 2 T^{\alpha}\|f_{n(T)}\|_\theta.
\]
On the other hand, by construction, 
\[
\| f_{n(T)} \|_{L^1} \geqslant \frac{1}{h\, T_k^{\xi}} \text{ and } \| f_{n(T)} \|_\theta \leqslant h \,T_k^{\theta\xi}
\]
for a constant $h=h(r)>1$ and for all $T_k \leqslant T < T_{k+1}$. 

It follows that, for $\mu$-almost every $v\in X$ and all $T$ sufficiently large, 
\[
\int\limits_0^T f_{n(T)}(\phi_t v) \, dt \geqslant \frac{1}{2h}T^{1-\xi} - 2 h T^{\alpha+\theta\xi}.
\]

If $1-\xi>\alpha+\theta\xi$, i.e., 
\[
\frac{1-\alpha}{1+\theta}>\xi,
\] the right-hand side of this inequality is strictly positive for all $T$ is sufficiently large. Since the function $f_{n(T)}$ is supported on $X_{T_k^{\xi}}$, we deduce that if $\frac{1-\alpha}{1+\theta}>\xi$
then, for $\mu$-almost every $v\in X$ and all $T$ sufficiently large, $\phi_{t_0}(v)\in X_{T_k^{\xi}}$ (where $T_k \leqslant T<T_{k+1}$) for some $0 \leqslant t_0 \leqslant T$. 

By plugging this information into Proposition \ref{p.depth}, we conclude that, if $\frac{1-\alpha}{1+\theta}>\xi$
then, for $\mu$-almost every $v\in X$ and all $T$ sufficiently large, 
\[ 
\delta(\phi_t(v)) \leqslant \frac{2}{T_k^{\xi/r}} \leqslant \frac{4}{T^{\xi/r}}
\]
for some time $0 \leqslant t_1 \leqslant T+a$ (where $a=a(r)$ is a constant). 

This proves the desired theorem: indeed, we can take the parameter $\xi$ arbitrarily close to $1/2$ in the previous paragraph because $\frac{1-\alpha}{1+\theta}\to 1/2$ as $\alpha \to 1/2$ and $\theta \to 0$. 

\end{proof}

\subsection{Very deep excursions are atypical}

We will now show by an elementary application of the Borel--Cantelli lemma that a typical geodesic \emph{doesn't} perform very deep excursions: 

\begin{theorem}\label{t.upper} For $\mu$-almost every $v \in T^1 S$, any $\kappa > 0$ and for all $T$ sufficiently large (depending on $v$ and $r>2$), 
\[
\delta(\phi_t(v)) >T^{-\frac{1}{r}-\kappa}
\]
for all  times $0 \leqslant t \leqslant T$. 
\end{theorem}

\begin{proof} Let $\xi>0$ and $\beta>0$ be parameters to be chosen later, and denote $T_k=k^{\beta}$. 

By the convexity of the distance $\delta$ to the cusp i.e., by Lemma \ref{l.convex}, 
we see that if $\delta(w) =T_k^{-\xi}$, then $\delta(\phi_s(w)) \in [(1/2)T_k^{-\xi}, 2T_k^{-\xi}]$ for all $|s| \sim T_k^{-\xi}$. 

Therefore, if we divide $[0,T_k]$ into $\sim T_k^{1+\xi}$ intervals $I_j^{(k)}=[a_j^{(k)}, b_j^{(k)}]$ of sizes $\sim T_k^{-\xi}$, then 
\begin{align*}
&\left\{v : \exists \, t \in I_j^{(k)} \textrm{ with } \delta(\phi_t(v)) = T_k^{-\xi} \right\} \\
\subset &\left\{v : \delta(\phi_{a_j^{(k)}}(v))\in \left[\frac{1}{2}T_k^{-\xi}, 2T_k^{-\xi}\right]\right\}.
\end{align*} 

By \eqref{e.volume}
\[
\mu \big( \mathbb{D}^\ast (2/R) \setminus  \mathbb{D}^\ast (1/2R) \big) = O\left( \frac{1}{R^{r+1}} \right).
\]

Since the Liouville measure $\mu$ is $\phi_t$-invariant we deduce that  
\[
\mu\left(\{v\in T^1S : \exists \, t\in I_j^{(k)} \textrm{ with } \delta(\phi_t(v)) = T_k^{-\xi} \} \right)=O \left( \frac{1}{T_k^{\xi(r+1)}} \right)
\]
for all $j$. Because we need $\sim T_k^{1+\xi}$ indices $j$ to cover the time interval $[0,T_k]$, we obtain that 
\begin{equation}\label{e.boundary}
\begin{split}
&\mu\left (\{v\in T^1S: \exists \, t\in [0,T_k] \textrm{ with } \delta(\phi_t(v)) = T_k^{-\xi} \} \right)  \\
&= O \left(\frac{1}{T_k^{\xi r-1}} \right).
\end{split} 
\end{equation}

We want to study the set 
\[
A_k=\left\{v\in T^1S: \exists \, t\in [0,T_k] \textrm{ with } \delta(\phi_t(v)) \leqslant T_k^{-\xi} \right\}.
\] 
We divide $A_k$ into $B_k:=A_k\cap\{v\in T^1S: \delta(v)\leqslant 2T_k^{-\xi}\}$ and $C_k:=A_k\setminus B_k$. Because we have $\mu(B_k) \leqslant \mu(\{v\in T^1S: x(v)\leqslant 2T_k^{-\xi}\}) = O(1/T_k^{\xi(r+1)})$, we just need to compute $\mu(C_k)$. To compute this we observe that 
\[
C_k\subset \left\{v: \exists \, t\in [0,T_k] \textrm{ with } \delta(\phi_t(v)) = T_k^{-\xi} \right\}
\]
and, \emph{a fortiori}, $\mu(C_k)=O(1/T_k^{\xi r-1})$ thanks to \eqref{e.boundary}. In particular, 
\begin{align*} 
\mu\left(\{v\in T^1S: \exists \, t\in [0,T_k] \textrm{ with } \delta(\phi_t(v)) \leqslant T_k^{-\xi}\}\right) &= \mu(A_k) \\
&= O \left( \frac{1}{T_k^{\xi r-1}} \right).
\end{align*} 

Note that the series $\sum\limits_{k=1}^{\infty}1/T_k^{\xi r - 1} = \sum\limits_{k=1}^{\infty}1/k^{\beta(\xi r - 1)}$ is summable when $\beta(\xi r-1)>1$, i.e., when $\xi>\frac{1}{r}(1+\frac{1}{\beta})$.
In this context, Borel--Cantelli lemma implies that, for $\mu$-almost every $v\in T^1S$, we have $\delta(\phi_t(v))>T_k^{-\xi}$ for all $t\in [0, T_k]$ and all $T_k=k^{\beta}$ sufficiently large (depending on $v$). Since $\frac{T_{k+1}}{T_k}\to 1$ as $k\to\infty$, we conclude that if 
\[
\xi>\frac{1}{r}(1+\frac{1}{\beta})
\]
then for $\mu$-almost every $v\in T^1S$, the distance to the cusp satisfies $\delta(\phi_t(v))>T^{-\xi}$ for all $t\in [0, T]$ and all $T$ sufficiently large (depending on $v$). 

This ends the proof of the theorem: in fact, by letting $\beta\to\infty$, we can take $\xi>1/r$ arbitrarily close to $1/r$ in the previous paragraph. 

\end{proof} 

\begin{remark} The upper bound for very deep cusp excursions in this theorem does not need the effective ergodic theorem and/or speed of mixing. In particular, we will see in Section \ref{s.applications} an extension of this result to the case of Weil-Petersson geodesic flows in non-exceptional moduli spaces (whose precise rates of mixing are currently unknown). 
\end{remark}

\section{Winding numbers and hyperbolic distances}\label{s.winding}

In this section, we give estimates for the statistics of winding numbers around cusps and hyperbolic distances along random Weil-Petersson type geodesics. 

\subsection{Winding number during an excursion} We will first describe how to calculate the winding number in an excursion. Let $\phi_t(v)$ be a Weil-Petersson type geodesic with an excursion in a cusp neighbourhood. Suppose that during the excursion the distance to the cusp satisfies $\delta(\phi_t(v)) \geqslant 1/D$ with $\delta_{\text{min}} = 1/D$.

By the quasi-Clairaut relations, 
\[
\frac{1}{M} \delta_{\text{min}}^r \leqslant \delta(\phi_t(v))^r b(\phi_t(v)) \leqslant M \delta_{\text{min}}^r
\]
which is equivalent to $1/MD^r \leqslant  \delta(\phi_t(v))^r b(\phi_t(v)) \leqslant M/D^r$. 

For $0 < B \leqslant \delta_0$, recall that $\mathbb{D}^\ast (B)$ is the cusp neighbourhood given by $\delta \leqslant B$. Let $S^1(B) = \partial \mathbb{D}^\ast (B)$ i.e., it is the level set $\delta = B$. In what follows, we write $x \asymp y$ if there exists a constant $C(r)> 1$ that depends only on $r$ such that $C(r)^{-1} x \leqslant y \leqslant C(r) x$. So for instance, the quasi-Clairaut relations will be written as
\[
 \delta(\phi_t(v))^r b(\phi_t(v)) \asymp \delta_{\text{min}}^r.
\]

\begin{lemma}\label{l.level}
For all $B$ with $0 < B \leqslant \delta_0$, the length of $S^1(B)$ satisfies
\[
\ell(S^1(B)) \asymp B^r.
\]
\end{lemma}

\begin{proof}
From \cite[Page 260]{BMMW}, the arc-length element along $S^1(B)$ satisfies 
\[
d \ell = \left( \frac{B^r}{\delta_0^r}  + O(B^{r+1}) \right)\, db.
\]
The lemma then follows by integrating the arc-length element.

\end{proof} 

Let $B = \delta(\phi_t(v))$. Recall that the cuspidal angular function $b(\phi_t(v)) = \langle \phi_t(v), JV \rangle$ measures the projection of the velocity vector to the direction $JV$ that is tangent to the level set $\delta= B$. By Lemma \ref{l.level}, the length $\ell(S^1(B))$ of the level set is $\asymp B^r$ and hence the winding number $w$ along the excursion is given by 
\begin{equation}\label{e.windint}
w \asymp \int \frac{b(\phi_t(v))}{\delta(\phi_t(v))^r} \, dt. 
\end{equation}
\begin{remark}
By \cite[Lemma 3.1 (1)]{BMMW}, we have
\[
\frac{d \delta(\phi_t(v))}{dt}  = a = \sqrt{1- b^2}.
\]
For an excursion with $\delta_{\text{min}} = 1/D$, the quasi-Clairaut relations give $b \asymp 1/\delta^r D^r$. Substituting the pair of observations in to \ref{e.windint}, we get 
\[
w \asymp \int \frac{1}{D^r \delta^{2r} \sqrt{1- b^2}} \, d\delta.
\]
For the surface of revolution of $y = x^r$, the Clairaut relations $\delta^r b = \delta_{\text{min}}^r = 1/D^r$ and a further change of variable $\delta \to u$ reduces this exactly to \cite[Equation (8)]{Pol-Wei}. 
\end{remark}

\begin{proposition}\label{p.winding}
Suppose a Weil--Petersson type geodesic $\phi_t(v)$ has an excursion in a cusp neighbourhood with $\delta_{\text{min}} = 1/D$ for some $D > 0$ large enough. Then the winding number $w$ for $\phi_t(v)$ corresponding to the excursions satisfies
\begin{equation}\label{e.wind}
w \asymp D^{r-1}.
\end{equation}
\end{proposition}

\begin{proof}
Recall that $v$ is the initial vector for the geodesic $\phi_t(v)$. We will formulate Equation \ref{e.windint} entirely in terms of the variable $b$ and then give an estimate for the integral.

By \cite[Lemma 3.1 (3)]{BMMW}
\[
\frac{d b(\phi_t(v))}{dt} = - \frac{rab}{\delta} + O(ab).
\]
Since during the excursion $b(\phi_t(v)) > 0$ and $a(\phi_t(v)) \leqslant 0$, we can substitute $a = - \sqrt{1-b^2}$ to get
\[
\frac{d b(\phi_t(v))}{dt} = \frac{r b\sqrt{1-b^2}}{\delta} - O(b \sqrt{1-b^2}).
\]
Equation \eqref{e.windint} then becomes
\[
w \asymp \int\limits_{b(v)}^1 \left( \frac{b}{\delta^r} \right) \left( \frac{\delta}{rb\sqrt{1-b^2} -\delta O(b\sqrt{1-b^2})} \right) \, db.
\]
By using the quasi-Clairaut relations, we get
\begin{align*} 
w &\asymp  \int\limits_{b(v)}^1 \frac{(bD^r)^{(r-1)/r}}{r \sqrt{1-b^2} - D^{-1} O(b^{-1/r} \sqrt{1-b^2})} \, db \\
&= D^{r-1} \int\limits_{b(v)}^1 \frac{b}{r b^{1/r}\sqrt{1-b^2} - D^{-1} O( \sqrt{1-b^2})} \, db.
\end{align*} 
Rearranging the integral on the right as 
\[
\int\limits_{b(v)}^1 \frac{b}{r b^{1/r}\sqrt{1-b^2} - D^{-1} O(\sqrt{1-b^2})} \, db \\
= \int\limits_{b(v)}^1 
\left( \frac{1}{rb^{1/r} - D^{-1}O(1)} \right) \left( \frac{b}{\sqrt{1-b^2}}\right) \, db
\]
it remains to show that the integral is bounded above and below independent of $D$. 

To simplify notation, we set $c = b(v)$. By the quasi-Clairaut relations, $ c\asymp 1/D^r \delta_0^r$. Let $J$ be the positive integer that satisfies $c2^J \leqslant 1 \leqslant c2^{J+1}$. We cover the interval $[c, 1]$ dyadically by intervals $I_j$ for $j = 0,1, \cdots, J-1$ where for $j \leqslant J-2$ 
\[
I_j = [c 2^j, c2^{j+1}]
\]
and $I_{J-1} = [c2^{J-1}, 1]$. We will split our integral over these sub-intervals $I_j$ and estimate each individually.

Note that $1/4 \leqslant c2^{J-1} \leqslant 1/2$. We can then conclude that on the interval $I_{J-1}$ 
\[
\frac{1}{r/2 - D^{-1} O(1)} \leqslant \frac{1}{rb^{1/r} - D^{-1}O(1)} \leqslant \frac{1}{r/4 - D^{-1} O(1)}. 
\]
For $D$ sufficiently large, the left hand side is bounded below by $2/r$ and the right hand side is bounded above by $8/r$. As a consequence, 
\[
\frac{2}{r} \int\limits_{c2^{J-1}}^{1} \frac{b}{\sqrt{1-b^2}} \, db \leqslant \int\limits_{c2^{J-1}}^{1} \left( \frac{1}{rb^{1/r} - D^{-1}O(1)} \right) \left( \frac{b}{\sqrt{1-b^2}} \right) \, db \\
\leqslant \frac{8}{r} \int\limits_{c2^{J-1}}^{1} \frac{b}{\sqrt{1-b^2}} \, db. 
\]
The anti-derivative of $b / \sqrt{1-b^2}$ is $-\sqrt{1-b^2}$. So by direct integration, the left hand side is $(2/r) \sqrt{1-c2^{J-1}} $ which is bounded below by $(2/r)\sqrt{1- 1/2}$. Similarly, by direct integration, the right hand side is $(8/r) \sqrt{1-c2^{J-1}}$ which is bounded above by $(8/r) \sqrt{1-1/4}$. 
In summary, the integral in question restricted to $I_{J-1}$ is bounded.

We now estimate the integral 
\[
A_j = \int\limits_{c2^j}^{c2^{j+1}} \left( \frac{1}{rb^{1/r} - D^{-1}O(1)} \right) \left( \frac{b}{\sqrt{1-b^2}}\right) \, db 
\]
restricted to the dyadic interval $[c2^j, c2^{j+1}]$ for $j \leqslant J-2$. On the interval $[c2^j, c2^{j+1}]$, we have
\[
\frac{1}{r c^{1/r} 2^{(j+1)/r} - \delta_0 c^{1/r} O(1)} \leqslant \frac{1}{rb^{1/r} - D^{-1}O(1)} \leqslant
\frac{1}{r c^{1/r} 2^{j/r} - \delta_0 c^{1/r} O(1)}
\]
Thus if we let
\[
B_j = \int\limits_{c2^j}^{c2^{j+1}} \frac{b}{\sqrt{1-b^2}} \, db
\]
then
\[
\left( \frac{1}{r c^{1/r} 2^{(j+1)/r} - \delta_0 c^{1/r} O(1)} \right) B_j \leqslant A_j \leqslant \left( \frac{1}{r c^{1/r} 2^{j/r} - \delta_0 c^{1/r} O(1)} \right) B_J.
\]
Equivalently
\[
A_j \asymp \left( \frac{1}{r c^{1/r} 2^{j/r} - \delta_0 c^{1/r} O(1)} \right) B_j.
\]
By direct integration
\[
B_j = -\sqrt{1-b^2} \, \Big|_{c2^j}^{c2^{j+1}} \asymp \sqrt{1 - c^2 2^{2j}} - \sqrt{1 - c^2 2^{2j+2}} \asymp (c2^j)^2.
\]
Thus
\[
A_j \asymp (c 2^j)^{(2r-1)/r}.
\]
In conclusion, we get the estimate
\[
w \asymp D^{r-1} \sum A_j \asymp D^{r-1} (c2^J)^{(2r-1)/r} \asymp D^{r-1}O(1).
\]
This concludes the proof of the proposition.

\end{proof}

\subsection{Total winding along random geodesics}
Combining the estimate in Proposition \ref{p.winding} for the winding number during an excursion  with the effective ergodic theory, we now estimate the total winding along random Weil--Petersson type  geodesic.

\subsubsection{The winding function:} 
As before let $p$ be a smooth non-negative bump function that is 1 on $\delta_0 - d/2 \leqslant \delta \leqslant \delta_0 + d/2$ and supported on $\delta_0 - d \leqslant \delta \leqslant \delta_0+ d$ such that $\Vert p \Vert_{C^1} \leqslant 10/d$. Recall the cuspidal angular function $a(v) = \langle v, V \rangle$. Let $V$ be the set of $v$ be such that $\delta_0 - d \leqslant \delta (v) \leqslant \delta_0+ d$ and $-1 < a(v) < 0$. In other words, $V$ consists of unit tangent vectors which under geodesic flow immediately head for an excursion deeper into the cusp neighbourhood. Let $w$ be the smooth function supported on $V$ defined by
\[
w(v) = p(v)\cdot ( \text{winding number during the excursion defined by the geodesic } \phi_t(v) ).  
\]
For $R >1$, let $\overline{q}_R$ be a smooth non-negative bump function that is equal to 1 on $1/R \leqslant  b(v)$ and supported on $1/2R \leqslant b(v)$ such that $\Vert \overline{q}_R \Vert_{C^1} \leqslant 3R$. Let $w_R$ be the truncation of $w$ defined by 
\[
w_R  = \overline{q}_R \cdot w.
\]

\begin{lemma}\label{l.totalwind}
There exists a constant $\eta > 1$ such that for all $R> 1$ large enough
\[
\frac{1}{\eta} < \int_S w_R \, d\mu  < \eta
\]
and
\[
\Vert w_R \Vert_\theta \leqslant \eta R^{\theta + (r-1)/r }.
\]
\end{lemma}

\begin{proof}
It follows from Propositions \ref{p.depth} and \ref{p.winding} that 
\[
d \int\limits_{1/R}^{1} \frac{1}{y^{(r-1)/r}} \, dy < \int_S w_R \, d\mu  < 2d \int\limits_{1/2R}^{1} \frac{1}{y^{(r-1)/r}}
\]
which by direct integration implies that 
\[
\int_S w_R \, d\mu   = O(1).
\]
On the other hand, 
\[
\max\limits_{ v \in V} \, w_R(v) \asymp R^{(r-1)/r} 
\]
which is attained when $b(v) \asymp 1/R$. Noting the supports of bump functions $p$ and $\overline{q}$ we conclude that
\[
\Vert w_R \Vert_\theta \asymp \frac{R^{(r-1)/r} }{(1/R)^\theta} = R^{\theta + (r-1)/r }
\]
finishing the proof of the lemma.

\end{proof}

For a geodesic $\phi_t(v)$ and a fixed cusp neighbourhood, we define the \emph{total winding} $W_v(T)$ till time $T$ to be the sum of all winding numbers corresponding to all excursions in the cups neighbourhood till time $T$. More precisely,
\[
W_v(T) = \int\limits_0^T w(\phi_t(v)) \, dt.
\]

The main result in this section is as follows:

\begin{theorem}\label{t.winding}
There is a constant $P> 1$ such that for $\mu$-almost every $v$ and for all $T$ sufficiently large (depending on $v$ and $r$) the total winding till time $T$ satisfies
\[
\frac{1}{P} T < W_v(T) < PT.
\]
\end{theorem}

\begin{proof}
The upper bound for $W_v(T)$ directly follows from Lemma \ref{l.totalwind} and the ergodicity of Weil--Petersson flow. 

By Theorem \ref{t.lower}, for $\mu$-almost every $v$, for any $\kappa >0$ and all sufficiently large $T$ (depending on $v$ and $r$) 
\[
\delta(\phi_t(v)) \leqslant T^{-\frac{1}{2r} + \kappa}
\]
for some time $0 \leqslant t \leqslant T$. In particular, invoking Proposition \ref{p.depth} we can use the truncation $w_{T^{\frac{1}{2} - \kappa r}}$ to establish a lower bound for $W_v(T)$. The rest of the argument uses the lower bound in the effective ergodic theorem viz. Theorem \ref{t.effective}. 

As required for Theorem \ref{t.effective}, we let $1/2 < \alpha <1$ and $T_j = j^{2\alpha/(2\alpha -1)}$. We consider the function $n: \mathbb{R} \to \mathbb{N}$ given by $n(T) = (T_j)^{\frac{1}{2}- \kappa r}$ for each $T_j \leqslant T < T_{j+1}$. By Theorem \ref{t.effective}, given $m >1$ we have that for $\mu$-almost every $v$
\[
\frac{1}{m} T \Vert w_{n(T)} \Vert_{L^1} - 2T^{\alpha} \Vert w_{n(T)} \Vert_\theta \leqslant W_v(T).
\]
To establish the theorem, we need to only analyse the second term on the left. Note that
\[
T^\alpha \Vert w_{n(T)} \Vert_\theta \asymp T^\alpha T_j^{(\frac{1}{2} - \kappa r)(\theta + (r-1)/r)} \asymp T^{\upsilon}
\]
where 
\[
\upsilon = \alpha + \frac{1}{2} \theta + \frac{r-1}{2r} - \kappa r \theta - \kappa(r-1).
\]
Since 
\[
\upsilon < \alpha + \frac{1}{2} \theta + \frac{r-1}{2r}
\]
it suffices to prove that we can choose $\alpha$ and $\theta$ so that the right hand side above is less than 1. First, note that by making $\theta \to 0$ we can make the contribution of the second term as small as we want. In particular, we choose $\theta$ small enough such that 
\[
\frac{1}{2} \theta + \frac{r-1}{2r} < \frac{1}{2}.
\]
In particular, this means that as $\alpha \to (1/2)^+$, we can achieve
\[
\alpha + \frac{1}{2} \theta + \frac{r-1}{2r} < 1.
\]
This concludes the proof of the theorem.

\end{proof} 

\subsection{Hyperbolic distance along random Weil--Petersson geodesics:} The previous discussion about winding numbers can be easily adapted to give the order of magnitude of the hyperbolic distance travelled by a random Weil--Petersson geodesic in exceptional moduli spaces. Recall that for the Weil--Petersson metric on exceptional moduli spaces $r=3$. 

More concretely, let $\phi_t(v)$, $0\leqslant t\leqslant t_{\min}$, be a unit-speed Weil--Petersson geodesic reaching its closest position to the cusp at time $t_{\min}$. Denote by $\delta_{\min}$ the Weil--Petersson distance between $\phi_{t_{\min}}(v)$ and the cusp. 

Consider the upper half-plane model $\{z\in\mathbb{C}: |\textrm{Re}z|\leqslant 1/2, \textrm{Im}z \geqslant h_0\}$, $h_0\geqslant 1$, of cusp neighbourhoods in exceptional moduli spaces. In this model, $\phi_t(v)$ describes a curve $(x(t),y(t))$ whose tangent vector $v(t)$ has almost unit size with respect to the metric $(dx^2+dy^2)/y^3$. Therefore, the Euclidean size of $v(t)$ is $\sim y(t)^{3/2}$ and, \emph{a fortiori}, the hyperbolic length of $v(t)$ is $\sim y(t)^{1/2}$. Since $y(t)\sim \delta(\gamma(t))^{-2}$, we get that the hyperbolic distance $\textrm{dist}_{hyp}$ travelled by $\gamma(t)$ is 
$$\textrm{dist}_{hyp}\asymp \int_0^{t_{\min}}\frac{dt}{\delta(\phi_t(v))}.$$ 
This expression is somewhat similar to the formula \eqref{e.windint} and, in fact, one can use the arguments in the proof of Proposition \ref{p.winding} to show that 
$$\textrm{dist}_{hyp}\asymp \int_0^{t_{\min}}\frac{dt}{\delta(\phi_t(v))}\asymp\log D$$ 
where $\delta_{\min}=1/D$. 

By replacing Proposition \ref{p.winding} by this estimate in the arguments establishing Theorem \ref{t.winding}, we conclude that there is a constant $P>1$ such that, for almost every vector $v$ and for all $T$ sufficiently large, the Weil--Petersson geodesic segment $\{\phi_t(v)\}_{t\in[0,T]}$ (in an  exceptional moduli space) generated by $v$ travels a total hyperbolic distance $\textrm{dist}_{hyp}(T)$ with  
$$\frac{1}{P}T < \textrm{dist}_{hyp}(T) < P T.$$ 

\subsubsection{Comparison with the Sullivan law:} By considering the largest excursion till $T$ in the upper half-plane co-ordinates above, Theorems \ref{t.lower} and \ref{t.upper} imply that for some constant $c_1>1$ the maximum imaginary part $y_{\textrm{max}}^{\textrm{WP}}(T)$ satisfies 
\[
\frac{1}{c_1} T^{(1-\epsilon)/3} < y_{\textrm{max}}^{\textrm{WP}}(T) < c_1 T^{(2+ \epsilon)/3}.
\]
for all times $T$ large enough (depending on the Weil--Petersson geodesic). 

On the other hand, for the hyperbolic metric, by considering the largest excursion till time $T$ in the upper half-plane co-ordinates, the maximum imaginary part $y_{\textrm{max}}^{hyp}(T)$ satisfies
\[
\frac{1}{c_2} T^{1- \epsilon} < y_{\textrm{max}}^{hyp}(T) < c_2 T^{1+ \epsilon}
\]
for some constant $c_2 >1$ and for all times $T$ large enough (depending on the hyperbolic geodesic). These estimates can be proved by using similar techniques for the hyperbolic flow. See for example, \cite[Lemma 5.6]{Gad}.

Since the hyperbolic distance travelled by a random Weil--Petersson geodesic grows linearly in $T$, this indicates that the largest excursions are shallower for the Weil--Petersson metric.

\section{Effective ergodic theory of polynomially mixing flows and applications}\label{s.applications}

In this section, we derive an effective ergodic theorem for slow mixing, specifically polynomially mixing flows. Subsequently, we will comment on the applications of the theorem.

Let $g_t$ be a flow on a space $Y$ preserving a probability measure $\mu$. The flow $g_t$ has a polynomial type decay of correlations if there are constants $K, C> 0$ such that 
\begin{equation}\label{e.correlation-poly}
\left\vert \int\limits_Y f_1 f_2 \circ g_t \, d\mu - \int\limits_Y f_1 \, d\mu \int\limits_Y f_2 \, d\mu \right\vert \leqslant K (1+t)^{-C} \Vert f_1 \Vert_B \Vert f_2 \Vert_B 
\end{equation}
for all $t \geqslant 0$ and all real valued observables $f_1, f_2$ in a Banach space $B \subset L^1(\mu)$ consisting of all functions satisfying some regularity hypothesis. Polynomial mixing can be found in many non-uniformly hyperbolic flows. Examples include semi-intermittency type semi-flows and many billiards maps. See \cite{Mel}, \cite{Che-Zha} for details. Instead of dealing separately with decay of correlations of the form $(\log T)^A/ T^C$ for large $T$ we simplify the discussion by making the constant $C$ slightly smaller allowing us to drop the logarithmic numerator.

Analogously to Lemma \ref{l.1}, we get

\begin{lemma}\label{l.2}
There exists a constant $Q>0$ such that any observable $ f \in B$ with $\int_Y f \, d\mu = 0$ satisfies for all $T$ large enough one of the following bounds:
\begin{align}
\int\limits_Y \left(\int\limits_0^T f(g_t y) \, dt \right)^2\,d\mu(y) &\leqslant QT \|f\|_B^2 \text{   when } C> 1, \\
\int\limits_Y \left(\int\limits_0^T f(g_t y) \, dt \right)^2\,d\mu(y) &\leqslant  QT \log T \|f\|_B^2 \text{   when } C=1, \\
\int\limits_Y \left(\int\limits_0^T f(g_t y) \, dt \right)^2\,d\mu(y) &\leqslant  QT^{2-C} \|f\|_B^2 \text{   when } C< 1.
\end{align}
\end{lemma}

\begin{proof}
By the same reasoning as in the proof of Lemma \ref{l.1}, we get the bound
\[
\int\limits_Y \left( \int\limits_0^T f(g_t x) \right)^2 \, d\mu(x) \leqslant K \Vert f \Vert_B^2 \int\limits_0^T \int\limits_0^T (1+ | t- s |)^{-C} \, dt \, ds. 
\]
The lemma follows by direct integration. 
\[
\int\limits_0^T \int\limits_0^T (1+ |t-s|)^{-C} \, dt \, ds = 2 \int_0^T \int\limits_s^T (1+ t -s)^{-C} \, dt \, ds.
\]
When $C \neq 1$, the integral reduces to 
\[
\int\limits_0^T \frac{1}{1-C}(1+t-s)^{1-C} \Big|_s^T  \, ds = \int\limits_0^T \frac{1}{1-C}(1+ T - s)^{1-C}  \, ds - \int\limits_0^T \frac{1}{1-C} \, ds.
\]
When $C > 1$, the first of the integrals is bounded above by $O(T)$.
When $C < 1$, the first of the integrals is $O(T^{2-C})$. The lemma follows in these cases.
When $C =1$, the integral reduces to
\begin{align*}
\int\limits_0^T \log (1+ t -s) \Big|_s^T \, ds &= -(1+T-s) \log(1+T-s)\Big|_0^T + (1+T-s) \Big|_0^T \\
&= (1+T) \log (1+T) - T
\end{align*}
which implies the lemma.

\end{proof} 

\begin{remark} This lemma generalises Lemma \ref{l.1}. 
\end{remark}

Lemma \ref{l.2} permits to extend the effective ergodic theorem viz. Theorem \ref{t.effective} as follows:
\begin{theorem}\label{t.effect-poly}
In addition to the hypothesis for Theorem \ref{t.effective}, if $\alpha$ satisfies
\begin{equation}\label{condition}
\alpha > \min \left( \frac{1}{2}, 1 - \frac{C}{2} \right)
\end{equation} 
then Theorem \ref{t.effective} holds: Given $m>1$, a function $n: \mathbb{R} \to \mathbb{N}$ such that $n(T) = n(T_k)$ for each $T_k \leqslant T < T_{k+1}$, and a sequence $\{f_j\}_{j \in \mathbb{N}} \subset B$ of non-negative functions, we have for $\mu$-almost every $y \in Y$ that
\begin{align*}
\frac{1}{m} T \| f_{n(T)} \|_{L^1} - 2 T^{\alpha}\| f_{n(T)}\|_{B} &\leqslant \int\limits_0^T f_{n(T)}(g_t y) dt \\ &\leqslant mT\| f_{n(T)} \|_{L^1} + 2 T^{\alpha}\| f_{n(T)} \|_B
\end{align*} 
for all $T$ sufficiently large (depending on $y$). 
\end{theorem}

\begin{proof}
Given $f \in B$, let $F =  f - \int_Y f \, d\mu$. 

If $C>1$ then the bound in Lemma \ref{l.2} implies that the proof of Theorem \ref{t.effective} works without change. If $C=1$ then the estimate \eqref{e.average} gets modified to
\[
\mu\left(\left\{y \in Y: \left(\int_0^T F(g_t y) \, dt \right)^2 \geqslant T^{2\alpha}\| f \|_{B}^2\right\}\right)\leqslant 4Q T^{1-2\alpha} \log T.
\]
Letting $F_j = f_j - \int_Y f_j \, d\mu$, the estimate above implies
\[
\mu \left( \left \{ y \in Y: \left(\int\limits_0^{T_k} F_{n(T_k)}(g_t y) \, dt \right)^2 \geqslant T_k^{2\alpha}\| f_{n(T_k)} \|_{B}^2\right\} \right) \\
\leqslant 4Q T_k^{1-2\alpha} \log T_k = 4Q \frac{2\alpha \log k}{(2\alpha-1)k^{2\alpha}}
\]
and 
\[
\mu \left( \left\{ y \in Y: \left(\int\limits_0^{T_{k+1}} F_{n(T_k)}(g_t y) \, dt \right)^2 \geqslant T_{k+1}^{2\alpha}\| f_{n(T_k)} \|_{B}^2 \right\} \right) \\
\leqslant 4Q \frac{2\alpha \log (k+1)}{(2\alpha-1)(k+1)^{2\alpha}}.
\]
The series $\sum_{i=1}^\infty \frac{\log i}{i^{2\alpha}}$ is summable for $\alpha > 1/2$. The remainder of the proof is then unchanged from Theorem \ref{t.effective}. 

If $C< 1$, then the estimate \eqref{e.average} gets modified to 
\[
\mu\left(\left\{y \in Y: \left(\int_0^T F(g_t y) \, dt \right)^2 \geqslant T^{2\alpha}\| f \|_{B}^2\right\}\right)\leqslant 4Q T^{2-C-2\alpha}.
\]
Letting $F_j = f_j - \int_Y f_j \, d\mu$, the estimate above implies
\[
\mu \left( \left \{ y \in Y: \left(\int\limits_0^{T_k} F_{n(T_k)}(g_t y) \, dt \right)^2 \geqslant T_k^{2\alpha}\| f_{n(T_k)} \|_{B}^2\right\} \right) \\
\leqslant 4Q T_k^{2-C-2\alpha} = 4Q k^{2\alpha(2-C-2\alpha)/(2\alpha -1)}
\]
and 
\[
\mu \left( \left\{ y \in Y: \left(\int\limits_0^{T_{k+1}} F_{n(T_k)}(g_t y) \, dt \right)^2 \geqslant T_{k+1}^{2\alpha}\| f_{n(T_k)} \|_{B}^2 \right\} \right) \\
\leqslant 4Q (k+1)^{2\alpha(2-C-2\alpha)/(2\alpha -1)}.
\]
The series $\sum_{i=1}^\infty i^{2\alpha(2-C-2\alpha)/(2\alpha -1)}$ is summable when 
\[
\frac{2\alpha(2-C-2\alpha)}{2\alpha -1} < -1
\]
which is satisfied when $2-C-2\alpha < 0$ or equivalently 
\[
\alpha > 1 - \frac{C}{2}.
\]
The rest of the proof now follows Theorem \ref{t.effective}.

\end{proof}

\subsection{Applications} 

Let us now briefly indicate how the general effective ergodic theorem (cf. Theorem \ref{t.effect-poly}) can be combined with our arguments from Sections \ref{s.excursions} and \ref{s.winding} to analyse certain statistics of random Weil--Petersson geodesics in non-exceptional moduli spaces and typical trajectories of planar billiards. 

\subsubsection{Bounds for cusp excursions for the Weil--Petersson metric on non-exceptional moduli spaces:} Consider the Weil--Petersson geodesic flow on the unit cotangent bundle $T^1\mathcal{M}_{g,n}$ of non-exceptional moduli spaces $\mathcal{M}_{g,n}$; $3g-3+n>1$. Recall that $v\in T^1\mathcal{M}_{g,n}$ is $v=(X,q)$ where $X$ is a Riemann surface and $q$ is a quadratic differential with unit Weil--Petersson size. 

The proof of Theorem \ref{t.upper} can be modified to give lower bounds on the hyperbolic systoles $\textrm{sys}(\phi_t(v))$ of the Riemann surfaces associated to a Weil--Petersson geodesic $\phi_t(v)$. 

More precisely, we affirm that, for almost every $v\in T^1\mathcal{M}_{g,n}$, for any $\kappa>0$, and for all $T$ sufficiently large (depending on $v$), the Weil--Petersson geodesic $\phi_t(v)$ generated by $v$ satisfies 
\begin{equation}\label{e.upper-WP-general}
\textrm{sys}(\phi_t(v)) > T^{-\frac{1}{3}-\kappa}
\end{equation}
for all times $0\leqslant t\leqslant T$.  

In fact, let $\xi>0$ and $\beta>0$ be parameters to be chosen later, and denote $T_k=k^{\beta}$. 

By \cite[Lemma 4.14]{BMW}, if $\textrm{sys}(w) =T_k^{-\xi}$, then we have $\delta(\phi_s(w)) \in [(1/2)T_k^{-\xi}, 2T_k^{-\xi}]$ for all $|s| \sim T_k^{-\xi}$. Hence, by dividing $[0,T_k]$ into $\sim T_k^{1+\xi}$ intervals $I_j^{(k)}=[a_j^{(k)}, b_j^{(k)}]$ of sizes $\sim T_k^{-\xi}$, we obtain  
\[
\left\{v: \exists \, t \in I_j^{(k)} \textrm{ with } \textrm{sys}(\phi_t(v)) = T_k^{-\xi} \right\} \\
\subset \left\{v: \textrm{sys}(\phi_{a_j^{(k)}}(v))\in \left[\frac{1}{2}T_k^{-\xi}, 2T_k^{-\xi}\right]\right\}.
\]

Let $\mu$ be the Weil--Petersson volume. By \cite[Lemma 6.1]{BMW},   
\[
\mu \big(\{w\in T^1\mathcal{M}_{g,n}: 1/2R \leqslant \textrm{sys}(w)\leqslant 2/R\} \big) = O\left( \frac{1}{R^{4}} \right).
\]

From the invariance of $\mu$, we deduce that  
\[
\mu\left(\{v\in T^1\mathcal{M}_{g,n}: \exists \, t\in I_j^{(k)} \textrm{ with } \textrm{sys}(\phi_t(v)) = T_k^{-\xi} \} \right)=O \left( \frac{1}{T_k^{4\xi}} \right)
\]
for all $j$. Since $\sim T_k^{1+\xi}$ indices $j$ are needed to cover the time interval $[0,T_k]$, we conclude that 
\begin{equation*}
\mu\left (\{v\in T^1\mathcal{M}_{g,n}: \exists \, t\in [0,T_k] \textrm{ with } \textrm{sys}(\phi_t(v)) = T_k^{-\xi} \}\right)=O \left( \frac{1}{T_k^{3\xi -1}} \right)
\end{equation*}

Define $A_k=\left\{v\in T^1\mathcal{M}_{g,n}: \exists \, t\in [0,T_k] \textrm{ with } \textrm{sys}(\phi_t(v)) \leqslant T_k^{-\xi} \right\}$ and divide it into $B_k:=A_k\cap\{v\in T^1\mathcal{M}_{g,n}: \textrm{sys}(v)\leqslant 2T_k^{-\xi}\}$ and $C_k:=A_k\setminus B_k$. Since $\mu(B_k) = O(1/T_k^{4\xi})$ (thanks to \cite[Lemma 6.1]{BMW}), our task is to estimate $\mu(C_k)$. We observe that 
\[
C_k\subset \left\{v: \exists \, t\in [0,T_k] \textrm{ with } \textrm{sys}(\phi_t(v)) = T_k^{-\xi} \right\}
\]
and, \emph{a fortiori}, $\mu(C_k)=O(1/T_k^{3\xi-1})$. In particular, 
\[
\mu\left(\{v\in T^1\mathcal{M}_{g,n}: \exists \, t\in [0,T_k] \textrm{ with } \textrm{sys}(\phi_t(v)) \leqslant T_k^{-\xi}\}\right) \\
=\mu(A_k) = O \left( \frac{1}{T_k^{3\xi -1}} \right)
\]

The series $\sum\limits_{k=1}^{\infty}1/T_k^{3\xi  - 1} = \sum\limits_{k=1}^{\infty}1/k^{\beta(3\xi  - 1)}$ is summable when $\beta(3\xi -1)>1$, i.e., when $\xi>\frac{1}{3}(1+\frac{1}{\beta})$. Hence,
the Borel--Cantelli lemma implies that, for $\mu$-almost every $v\in T^1\mathcal{M}_{g,n}$, we have $\textrm{sys}(\phi_t(v))>T_k^{-\xi}$ for all $t\in [0, T_k]$ and all $T_k=k^{\beta}$ sufficiently large (depending on $v$). Because $\frac{T_{k+1}}{T_k}\to 1$ as $k\to\infty$, we conclude that if 
\[
\xi>\frac{1}{3}(1+\frac{1}{\beta})
\]
then, for $\mu$-almost every $v\in T^1\mathcal{M}_{g,n}$, one has $\textrm{sys}(\phi_t(v))>T^{-\xi}$ for all $t\in [0, T]$ and all $T$ sufficiently large (depending on $v$). By letting $\beta\to\infty$, we get \eqref{e.upper-WP-general}. 

On the other hand, we are currently \emph{unable} to adapt the proof of Theorem \ref{t.lower} to get upper bounds on the systoles $\textrm{sys}(\phi_t(v))$ of the Riemann surfaces associated to a Weil--Petersson geodesic $\phi_t(v)$. Indeed, the proof of Theorem \ref{t.lower} relies on the effective ergodic theorem which in turn depends on polynomial rates of mixing. At present, it is unknown whether or not Weil--Petersson flows on non-exceptional moduli spaces have polynomial rates of mixing \cite{BMMW2}. Nevertheless, if we \emph{assume} that the Weil--Petersson flow on $T^1\mathcal{M}_{g,n}$ has a polynomial rate of mixing \eqref{e.correlation-poly} for a certain constant $C>0$, \emph{then} one can follow the steps in the proof of Theorem 4.5 to show that, for almost every $v\in T^1\mathcal{M}_{g,n}$, for any $\kappa>0$, and for all $T$ sufficiently large (depending on $v$), the Weil--Petersson geodesic $\phi_t(v)$ generated by $v$ would satisfy 
\begin{equation}\label{e.lower-WP-general}
\textrm{sys}(\phi_t(v)) \leqslant T^{-\frac{\min\{1, C\}}{8}-\kappa}
\end{equation}
for some time $0\leqslant t\leqslant T$. 

\subsubsection{Billiard dynamics} The investigation of the precise rates of mixing for several classes of non-uniformly hyperbolic planar billiard maps and flows is a well-developed subject with a vast literature. In particular, Chernov--Zhang \cite{Che-Zha} established almost linear rate of mixing (i.e., \eqref{e.correlation-poly} with $C<1$ arbitrarily close to one) for the maps associated to many semi-dispersing and Bunimovich-stadia billiards, Melbourne \cite{Mel} obtained similar estimates for the corresponding billiard flows, and Baladi--Demers--Liverani \cite{BDL} recently proved exponential mixing for finite horizon Sinai billiard flows. 

Therefore, we have that a wide class of planar billiards satisfy the effective ergodic theorem in Theorem \ref{t.effect-poly} with $C=1$. 

We expect that, thanks to the nice features of the so-called \emph{homogeneity strips} (see e.g. \cite{Che-Zha}), this effective ergodic theorem can be employed to analyse almost greasing collisions along random billiard trajectories. Moreover, we believe that our investigation of winding numbers can be modifed to study the sizes of \emph{corner series}\footnote{Sequence of consecutive short bounces near the cusp.} of random trajectories of the billiard flow associated to the three-cusps billiard considered by Chernov--Markarian \cite{Che-Mark}.


\end{document}